\documentclass{amsart}
 \usepackage{amsthm,amsfonts,amsmath,amssymb,latexsym,epsfig,upref,eucal,ae}
\usepackage[all]{xy}

\usepackage{graphicx}

\usepackage{subfigure}
\usepackage{epic}
\usepackage{eepic}
\usepackage{setspace}
\usepackage{booktabs}
\usepackage{longtable}
\usepackage{pstricks}
\usepackage{url}

\newtheorem{theorem}{Theorem}
\newtheorem{lemma}{Lemma}
\newtheorem{proposition}{Proposition}

\newtheorem{definition}{Definition}

\newcommand{\tr}{\operatorname{tr}}
\newcommand{\pr}{p}
\newcommand{\Id}{\operatorname{Id}}

\newcommand{\End}{\operatorname{End}}

\newcommand{\ad}{\operatorname{ad}}

\newcommand{\Aut}{\operatorname{Aut}}
\newcommand{\dbar}{\bar{\partial}}
\newcommand{\imag}{\mathop{{\fam0 {\textbf{i}}}}\nolimits}
\newcommand{\CC}{{\mathbb C}}
\newcommand{\PP}{{\mathbb P}}

\newcommand{\RR}{{\mathbb R}}

\renewcommand{\)}{\right)}
\newcommand{\vol}{\operatorname{vol}}

\newcommand{\defeq}{\mathrel{\mathop:}=} 

\newcommand{\surj}{\to\kern-1.8ex\to}

\newcommand{\lto}{\longrightarrow}
\newcommand{\lra}[1]{\stackrel{#1}{\longrightarrow}}

\newcommand{\cA}{\mathcal{A}}
\newcommand{\cC}{\mathcal{C}}

\newcommand{\cE}{\mathcal{E}}

\newcommand{\cJ}{\mathcal{J}}
\newcommand{\cJi}{\mathcal{J}^{i}}
\newcommand{\cK}{\mathcal{K}}

\newcommand{\cP}{\mathcal{P}}
\newcommand{\cF}{\mathcal{F}}

\newcommand{\cY}{\mathcal{Y}}

\newcommand{\cX}{\mathcal{X}}
\newcommand{\cG}{\mathcal{G}}
\newcommand{\cL}{\mathcal{L}}
\newcommand{\cO}{\mathcal{O}}

\newcommand{\Lie}{\operatorname{Lie}}

\newcommand{\LieG}{\operatorname{Lie} \cG}
\newcommand{\LieX}{\operatorname{Lie} \ctG}

\newcommand{\cH}{\mathcal{H}} 
\newcommand{\LieH}{\Lie\cH}

\def\a{\gamma}
\def\b{\beta}

\def\om{\omega}
\def\Om{\Omega}

\def\del{\partial}
\def\delb{\overline{\partial}}

\def\vol{\mathrm{vol}}

\def\Lie{\mathrm{Lie}}

\def\Id{\mathrm{Id}}

\def\cA{\mathcal{A}}
\def\cB{\mathcal{B}}
\def\cJ{\mathcal{J}}
\def\ctG{\widetilde{\mathcal{G}}}
\def\cG{\mathcal{G}}
\def\cH{\mathcal{H}}
\def\cP{\mathcal{P}}
\def\cT{\mathcal{T}}
\def\cE{\mathcal{E}}

\def\iK{\it{K}}

\def\fg{\mathfrak{g}}
\def\k{\mathfrak{k}}

\def\del{\partial}
\def\delb{\overline\partial}

\begin{document}

\title[Deformation and Coupled K\"ahler--Yang--Mills equations]
{Deformation of complex structures and the Coupled K\"ahler--Yang--Mills equations}
\author[M. Garcia-Fernandez]{Mario Garcia-Fernandez}
\author[C. Tipler]{Carl Tipler}
\address{EPFL SB MATHGEOM, MA B3 495 (Batiment MA), Station 8, 8000, CH-1015 Lausanne, Switzerland}
\address{D\'epartement de math\'ematiques,
Universit\'e du Qu\'ebec \`a Montr\'eal,
Case postale 8888, succursale centre-ville,
Montr\'eal (Qu\'ebec), H3C 3P8}
\email{mariogf@qgm.au.dk ; carl.tipler@cirget.ca}

\thanks{The first author was supported by QGM (Centre for Quantum Geometry of Moduli Spaces) funded by the Danish National Research Foundation and also by a grant of the Hausdorff Research Institute for Mathematics. The second author is supported by a CRM-ISM grant.}

\begin{abstract}
In this work we define a deformation theory for the Coupled K\"ahler--Yang--Mills equations in arXiv:1102.0991, generalizing work of Sz\'ekelyhidi on constant scalar curvature K\"ahler metrics. We use the theory to find new solutions of the equations via deformation of the complex structure of a polarised manifold endowed with a holomorphic vector bundle. We also study the deformations of the recent examples of Keller and T{\o}nnesen-Friedman.
\end{abstract}

\maketitle


\section{Introduction}
\label{secintro}

The Coupled K\"ahler--Yang--Mills equations (in short CKYM) were introduced in \cite{GF} as a new approach to the moduli problem for triples $(X,L,E)$, where $E$ is a holomorphic principal $G^c$-bundle over a polarised complex manifold $(X,L)$. Solutions of the equations
\begin{equation}\label{eq:CKYM0}
\begin{split}
\Lambda F & = z,\\
S_g - \alpha \Lambda^2 \tr F \wedge F & = c,
\end{split}
\end{equation}
given by a K\"ahler metric $g$ on $X$ with K\"ahler class $c_1(L)$ and a reduction $H$ of $E$ to a maximal compact subgroup $G \subset G^c$, can be seen as natural uniformisers of the complex geometry of $(X,L,E)$. Here, $S_g$ is the scalar curvature of the metric, $F$ stands for the curvature of the Chern connection of $H$ and $\Lambda$ denotes contraction with the K\"ahler form. The equations depend on a \emph{coupling constant} $\alpha \in \RR$, that determine the topological constant $c$ (for further details see Section \ref{sec:background}).

Being a generalization of the conditions of constant scalar curvature for a K\"ahler metric and Hermitian--Yang--Mills for a connection, the equations \eqref{eq:CKYM0} describe the zeros of a moment map for the hamiltonian action of a group $\ctG$ on a infinite dimensional space $\cP$. For fixed $\omega$ and $H$, $\cP$ parameterises structures of holomorphic principal bundle and K\"ahler manifold on the smooth manifolds underlying $E$ and $X$. The \emph{extended gauge group} $\ctG$ is given by a non trivial extension
$$
1\to \cG \lra{} \ctG \lra{} \cH \to 1
$$
of the group of Hamiltonian symplectomorphism for the symplectic structure $\omega$ by the gauge group of the reduction $H$. The coupling constant $\alpha$ is used in the definition of the symplectic structure on $\cP$, which is K\"ahler for $\alpha >0$. Relying on this moment map interpretation, a general theory for the study of the CKYM equations was developped by the first author jointly with Alvarez-Consul and Garcia-Prada in \cite{AGG}. Particular features of this problem emerge from the structure of $\ctG$, which does not admit a complexification nor an invariant metric on its Lie algebra (unlike $\cH$).

The difficulty, from the analytic viewpoint, in determining whether or not a solution of \eqref{eq:CKYM0} exists is that the resulting system of partial differential equations is coupled, fourth order and fully non-linear.
Despite this, some examples have been found in \cite{AGG,GF} for small $\alpha$, via perturbation of constant scalar curvature K\"ahler (cscK) metrics and Hermitian-Yang--Mills connections. More recently, examples of solutions over a polarised manifold which does not admit any cscK metric has been obtained by Keller and T{\o}nnesen--Friedman \cite{KT}. In this work we add to the list of examples, finding new solutions of \eqref{eq:CKYM0} via simultaneous deformation of the complex structures on $X$, $L$ and $E$ in \cite{KT}. The main tool to achieve this goal is a deformation theory for the CKYM equations, which
generalizes the one developed by Sz\'ekelyhidi \cite{sz} for cscK metrics. Roughly, the idea is that small complex deformations of $(X,L,E)$ which are stable in the GIT sense are the ones carrying solutions of the CKYM equations.

To state a precise result, we fix $(\omega,H)$, a solution of \eqref{eq:CKYM0} on $(X,L,E)$, and consider the group $K$ of those holomorphic automorphisms of $E$ lying in the extended gauge group $K \subset \ctG$. This group is compact and finite dimensional, and hence admits a complexification $K^c$. Using elliptic operator theory, in Section \ref{subsec:complex} we construct a finite dimensional representation
$$
K^c \curvearrowright H^1(X,L_\omega^*),
$$
parameterising infinitesimal deformations of the complex structure on $E$ and $X$. By application of the Kuranishi method in Section \ref{subsec:slice}, any small deformation $(X',L',E')$ of $(X,L,E)$ determines a point in $H^1(X,L_{\om}^*)$ (see Section \ref{sec:deformCKYM} for details).

\begin{theorem}
\label{theo:finitestab}
Let $(X,L,E)$ be a polarized K\"ahler manifold $(X,L)$ endowed with a holomorphic principal bundle $E\rightarrow X$, which admits a solution of~\eqref{eq:CKYM0} with $\alpha > 0$. Then, any small deformation $(X',L',E')$ with closed $K^c$-orbit in $H^1(X,L_{\om}^*)$ admits also a solution. In particular, if $K$ is finite the existence of solutions of~\eqref{eq:CKYM0} is an open condition under deformation of complex structure.
\end{theorem}

Our method of proof is based on that of Sz\'ekelyhidi in \cite{sz}. We consider an ambient space $\cP \subset \cJ \times \cA$ of almost complex structures and define an elliptic complex (Section \ref{sec:defopairs}), with first cohomology $H^1(X,L_\om^*)$, that encodes the structure of $\cJ\times \cA$ near $(X,L,E)$, taking into account the action of $\ctG$ and the subspace $\cP$. With this complex at hand, in Proposition \ref{prop:slice} we construct a slice in $\cJ\times \cA$ that parametrizes nearby structures up to complexified orbits (see Section \ref{subsec:orbits}). Using the slice, we pull-back the moment map problem to $H^1(X,L_\om^*)$ and apply Proposition \ref{propo:Gabor}, due to Sz\'ekelyhidi, which relates a moment map problem on a vector space with its linearisation. For this, we crucially need $\alpha > 0$ to assure that the pull-back symplectic structure is K\"ahler.

We apply Theorem \ref{theo:finitestab} in two different situations. First, to find new examples of solutions, we consider deformations of two different homogeneous split rank two bundle over the product $\PP_1 \times \PP_1$, with fixed base. One of the bundles corresponds to the reducible point in the moduli of ASD connections in \cite[p. 242]{DK}. Second, to illustrate the theory in the case where the polarized manifold is also deformed, we build on the non-trivial solutions of the coupled equations constructed by Keller and T{\o}nnesen-Friedman \cite{KT}. These solutions exhibit $\alpha >0$ and are on a line bundle over a polarised ruled threefold, that we deform using Theorem \ref{theo:finitestab} combined with Lemma \ref{lem:defX}. We should remark that the complex manifolds obtained by deformation have the same form as in \cite[Example 3.5]{KT} and, although we cannot prove it, it is likely that the solutions obtained applying Theorem \ref{theo:finitestab} are indeed covered in \cite{KT}. This would follow from the uniqueness of solutions of the CKYM equations, which is unknown up to date. The idea of applying Sz\'ekelyhidi's argument to find solutions of the equations at hand amounts to work of Donaldson in \cite{D3} (see also \cite{RoTi,Bro}).

Our application follows from a rather technical discussion in Section \ref{sec:ap}: we extend the elliptic complex in Section \ref{subsec:complexext} to a complex $L^*_\om$ and compare that to a standard complex $L^*$ encoding deformations of pairs $(X,E)$ (see Proposition \ref{prop:compare}), considered previously by Huang \cite{Hu} (see also \cite{Griff}). We expect this to be of independent interest, as in particular it settles all the tools necessary for the construction of the moduli space of solutions of \eqref{eq:CKYM0}. Following recent work of the first author with Alvarez-Consul and Garcia-Prada \cite{AGG2}, further motivation for the study of this moduli space comes from the Physics of cosmic strings. We hope to come back to this in future work.

\medskip

\textbf{Outline:} We start in Section \ref{sec:background} with the necessary background on the CKYM equations. In Section \ref{sec:defopairs} we construct our elliptic complex and the slice, using Kuranishi method. The proof of Theorem \ref{theo:finitestab} is contained in Section \ref{sec:deformCKYM}. Section \ref{sec:ap} considers the complexes $L_\omega^*$ and $L^*$ and ends in the proof of Proposition \ref{prop:compare}, that we apply to construct the examples in Section \ref{sec:example}.

\medskip

\textbf{Acknowledgments:} We wish to thank Luis Alvarez-Consul, Bjorn Andreas, Vestislav Apostolov, Paul Gauduchon, Julien Keller, Joan Pons-Llopis, Gabor Sz\'ekelyhidi and Christina T{\o}nnesen-Friedman for useful discussions. The first author wants to thank Luis Alvarez--Consul and Oscar Garcia-Prada for proposing the problem of the Coupled K\"ahler--Yang--Mills equations for his Thesis. Part of this work was undertaken
during the first author's visit to the Hausdorff Research Institute for Mathematics, and he wishes to thank the hospitality. The second author would like to thank both LMJL and CIRGET for providing stimulating and welcoming environments.



\section{Background on the CKYM equations}\label{sec:background}

Throughout this section, $M$ is a fixed compact symplectic manifold of dimension $2n$, with symplectic form $\omega$
and volume form $\vol_\omega =\frac{\omega^n}{n!}$.
We fix a real compact Lie group $G$ with Lie algebra $\fg$ and a smooth principal
$G$-bundle $H$ over $M$. We also fix
a positive definite inner product
$$
-\tr: \mathfrak{g} \times \mathfrak{g} \to \RR
$$
on $\mathfrak{g}$ invariant under the adjoint action, that induces a metric on the adjoint bundle $\ad H$. The
space of smooth $k$-forms on $M$ is denoted by $\Omega^{k}$ and for any associated vector bundle $F$, $\Omega^{k}(F)$ denotes the space of smooth
$F$-valued $k$-forms on $M$.
Considering the space $\Omega^k(\ad H)$, the metric on $\ad H$ extends to give a pairing
\begin{equation}\label{eq:Pairing}
\Omega^p(\ad H) \times \Omega^q(\ad H) \to \Omega^{p+q},
\end{equation}
that we write simply by $- \tr a_p \wedge a_q$ for $a_j \in \Omega^j(\ad H)$, $j = p,q$. Given an almost complex structure on $M$ compatible with $\omega$ it determines a metric and we can regard the operator
\begin{equation}
\label{eq:Lambda}
  \Lambda:\Omega^{p,q}\to \Omega^{p-1,q-1}
\end{equation}
acting  on the space of smooth $(p,q)$-forms as the adjoint of the Lefschetz operator $\Omega^{p-1,q-1}\to
\Omega^{p,q}\colon \gamma \mapsto \gamma
\wedge \omega$. Its linear extension
$\Lambda:\Omega^{p,q}(\ad H)\to \Omega^{p-1,q-1}(\ad H)$ will be denoted in the same way.

\subsection{The CKYM equations}\label{subsec:equations}

Let $\cJ$ and $\cA$ be the spaces of almost complex structures on $M$ compatible with $\omega$ and connections on $H$, whose elements we denote respectively by $J$ and $A$. In this section we explain how the CKYM equations arise naturally from the symplectic reduction of a subspace $\cP \subset \cJ \times \cA$ of integrable pairs.

The group of symmetries of this theory is the \emph{extended gauge group} $\ctG$. It is
the group  of automorphisms of $H$ which cover elements of the group of hamiltonian symplectomorphisms
$\cH$. There is a canonical short exact sequence of Lie groups
\begin{equation}
\label{eq:coupling-term-moment-map-1}
  1\to \cG \lra{} \ctG \lra{\pr} \cH \to 1,
\end{equation}
where the gauge group $\cG$ of $H$ is the normal subgroup $\cG\subset\ctG$ of automorphisms covering the identity. Here $\pr$ is the map that assigns to each automorphism the Hamiltonian symplectomorphisms that it covers.

There are actions of $\ctG$ on the spaces $\cJ$ and $\cA$, which combine to give an action on the product
\[
  g(J,A)=(\pr(g)J, gA), \text{ for $g\in\ctG$},
\]
where $\pr(g)J$ denotes the push-forward of $J$ by $\pr(g) \in \cH$. To describe the $\ctG$-action on $\cA$, we view the elements of $\cA$ as $G$-equivariant splittings $A\colon TH\to VH$ of the short exact sequence
\begin{equation}
\label{eq:principal-bundle-ses}
  0 \to VH \lto TH\lto \pi^*TM \to 0,
\end{equation}
where $VH=\ker d\pi$ is the vertical bundle. Using the action of $g\in \ctG$ on $TH$, the $\ctG$-action on $\cA$ is
$$
g A \defeq g\circ A \circ g^{-1}.
$$
Given a connection $A$, we use the notation $A^\perp$ for the horizontal lift of a vector field  $y$ on $M$ to a vector field $A^\perp y$ on $H$. Note that $A^\perp \eta_\phi \in \LieX$ for any Hamiltonian vector field $\eta_\phi \in \LieH$ given by a smooth function $\phi \in C^\infty(M)$. Hence, any connection determines a vector space splitting of the short exact sequence of Lie algebras induced by~\eqref{eq:coupling-term-moment-map-1}.


The spaces $\cJ$ and $\cA$ are endowed with $\ctG$-invariant symplectic structures $\omega_\cJ$ and $\omega_\cA$ induced by $\omega$, that we can combine to define a symplectic form on the product for each non-zero real constant $\alpha$, given by
\begin{equation}
\label{eq:Sympfamily}
\omega_\alpha = \omega_\cJ + \frac{4 \alpha}{(n-1)!} \omega_\cA.
\end{equation}
The following result provides the starting point for the theory of the CKYM equations.

\begin{proposition}\cite[Prop. 2.3.1]{GF}
\label{prop:momentmap-pairs} The $\ctG$-action on $(\cJ\times \cA,\omega_\alpha)$ is hamiltonian with $\ctG$-equivariant moment map $\mu_{\alpha}\colon \cJ \times \cA\to \LieX^*$ given by
\begin{equation} \label{eq:thm-muX}
\begin{split}
\langle \mu_{\alpha}(J,A),\zeta\rangle & = - 4\alpha\int_X \tr A\zeta\wedge(\Lambda F_A-z)\vol_{\omega}\\
&- \int_X  \phi\(S_J - \alpha \Lambda^2 \tr F_A \wedge F_A + 4\alpha \Lambda \tr F_A\wedge z - c_z\)\vol_{\omega}
\end{split}
\end{equation}
for any $\zeta\in\LieX$ covering $\eta_\phi \in \LieH$.
\end{proposition}

Here $z \in \mathfrak{g}$ is an element in the centre of $\mathfrak{g}$, $c_z$ is a real constant defined in \cite[Eq. (2.3.2)]{GF}, $S_J$ the Hermitian scalar curvature of $J$ and $F_A \in \Omega^2(\ad H)$ is the curvature of the connection $A$, given by
$$
F_A = - A[A^\perp\cdot,A^\perp\cdot].
$$
A key observation in \cite{GF} is that the space $\cJ \times \cA$ has a (formally integrable) complex structure preserved by the $\ctG$-action, given by
\begin{equation}
\label{eq:complexstructureI}
\mathbf{I}_{\mid(J,A)}(\gamma,a) = (J\gamma,-a(J \cdot)); \quad (\gamma,a) \in T_J\cJ\times T_A\cA.
\end{equation}
For $\alpha$ positive, $\mathbf{I}$ is compatible with the family of symplectic structures \eqref{eq:Sympfamily}, thus defining a K\"ahler structure on $\cJ \times \cA$. The condition $\alpha > 0$ will be assumed in the sequel.

Suppose now that there exist K\"{a}hler structures on $M$ with K\"{a}hler form $\omega$. This means that
the subspace $\cJi \subset \cJ$ of integrable almost complex structures compatible with $\omega$ is non-empty. For each $J\in \cJi$, let $\cA^{1,1}_J\subset\cA$ be the subspace of connections $A$ such that $F_A \in \Omega_J^{1,1}(\ad H)$. The \emph{space of integrable pairs}
\begin{equation}
\label{eq:cP}
  \cP\subset \cJ\times \cA,
\end{equation}
consisting of elements $(J,A)$ with $J\in \cJi$ and $A\in \cA^{1,1}_J$, is a (maybe singular) K\"ahler submanifold which inherits a Hamiltonian action of $\ctG$. The zero locus of the induced moment map $\mu_\alpha$ corresponds precisely to the solutions of the \emph{Coupled K\"ahler--Yang--Mills equations}
\begin{equation}\label{eq:CKYM}
\begin{split}
\Lambda F_A & = z,\\
S_J \; - \; \alpha \Lambda^2 \tr F_A \wedge F_A & = c.
\end{split}
\end{equation}
Here, $S_J$ i\def\d{\partial}s the scalar curvature determined by the metric $g_J = \omega(\cdot,J\cdot)$ and $c$ is a real constant.

\subsection{Complexified orbits}\label{subsec:orbits}

We recall now the notion of complexified $\ctG$-orbit (or simply $\ctG^c$-orbit) in $\cJ \times \cA$ \cite{AGG,GF}, that will be used in the application of Kuranishi's method in \S\ref{sec:defopairs}. This notion was introduced by Donaldson \cite{D1} for the action of $\cH$ on $\cJ$, who observed that while the complexification of $\cH$ does not exist one can make sense of its orbits as leaves of a distribution.

In the case we are concerned, the distribution is given by the image of the complexified infinitesimal action
\begin{equation}\label{eq:infactcpx}
P\colon \Lie\ctG^c \to T(\cJ \times \cA): \zeta_0 + \imag \zeta_1 \to P(\zeta_0) + \mathbf{I}P(\zeta_1),
\end{equation}
where $P(\zeta_j)$ denotes the infinitesimal action of $\zeta_j \in \Lie \ctG$.

\begin{definition}
We say that $(J_0,A_0)$ and $(J_1,A_1)$ are in the same $\ctG^c$-orbit if there exists paths $\zeta_t$ in $\Lie\ctG^c$ and $(J_t,A_t)$ in $\cJ \times \cA$ for $t \in [0,1]$ such that
\[
\frac{d}{dt}(J_t,A_t) = P(\zeta_t).
\]
\end{definition}

To perturb pairs in a $\ctG^c$-orbit, we need an analogue of the complex exponential map for $\Lie\ctG^c$ which extends the $\ctG$-action on $\cJ \times \cA$ \cite{AGG,GF}. To recall this notion, we need first some preliminary background. Let $G^c$ be the complexification of $G$ with Lie algebra $\mathfrak{g}^c$ and consider the associated principal $G^c$-bundle $P=H\times_G G^c$.  The use of the same notation for the infinitesimal action \eqref{eq:infactcpx} should lead to no confusion. Let $\Aut P$ the group of automorphisms of $P$, that is, the
$G^c$-equivariant diffeomorphisms $g\colon P\to P$. We can now regard the elements in $\cJ \times \cA$ as $G^c$-invariant almost complex structures on the principal $G^c$-bundle $P$ 
using the map
\begin{equation}
\label{eq:pairsinvacs}
(J,A) \mapsto I_{J,A} \defeq \imag A + A^{\perp}J \pi,
\end{equation}
where $\pi\colon P \to M$ is the projection and $\imag$ is the almost complex structure on the vertical bundle $VP \subset TP$ induced by the complex structure on $G^c$. This map is $\ctG$-equivariant and the condition $(J,A) \in \cP$ is equivalent to the integrability of $I_{J,A}$. 

Given $(J,A) \in \cJ \times \cA$, the \emph{complex exponential} \cite[Eq. (3.1.13)]{GF} is a map
$$
\exp_{J,A} \colon \LieX^c \to \Aut P
$$
that extends the exponential map on $\ctG$. To define it, let $\cK_J$ be the space of symplectic structures on $M$ compatible with $J$ and let $\cB$ be the space of pairs $(\omega',H')$ given by an element $\omega' \in \cK_J$ and a reduction $H'$ of $P$ to $G$. Given
$$
\imag\zeta = \imag(\xi + \theta_A^\perp \eta_\phi) \in \imag\LieX
$$
with $\xi \in \LieG$ and $\eta_\phi \in \cH$, we consider the curve
$$
(\omega + t dJd\phi,e^{it\xi}H) \in \cB,
$$
where we use the natural action of $\Aut P$ on the space of $G$-reductions of $P$. The complex exponential is now defined by
\begin{equation}\label{eq:exp}
\exp_{J,A}(\zeta' + \imag \zeta) := f_1\exp(\zeta'),
\end{equation}
for $\zeta' \in \LieX$ and $\zeta$ as before, where $f_t$ is the flow of the the time dependent vector field
$$
y_t = -I_{J,A}(\xi + A_t^\perp \eta_t)
$$
on $P$ with $f_0 = \Id$. Here, $\eta_t$ is the Hamiltonian vector field of $\phi$ with respect to the symplectic form $\omega + t dJd\phi$ and $A_t$ is the Chern connection of the reduction $e^{it\xi}H$ with respect to the complex structure $I_{J,A}$.

The complex exponential has some properties which are important for our construction in \S\ref{sec:defopairs}, that we recall now. Let $K$ the isotropy group of $(J,A)$ on $\ctG$ and note that $K$ acts on $\Aut P$ on the left. Then, $\exp_{J,A}$ is $K$-equivariant with respect to the adjoint action on $\LieX^c$. By construction, we also have a well defined \emph{complexified action}
\begin{equation}\label{eq:cpxact}
\LieX^c \to \cJ \times \cA \colon \zeta \to \exp_{J,A}(\zeta)^{-1} (J,A),
\end{equation}
where the image of $\zeta$ is the point in $\cJ \times \cA$ corresponding to the pull-back almost complex structure $(\exp_{J,A}(\zeta))^*I_{J,A}$. Furthermore, when $(J,A) \in \cP$ the image of the complexified action is contained in the $\ctG^c$-orbit of $(J,A)$ and
$$
\frac{d}{dt}_{|t=0}\exp_{J,A}(t\zeta)^{-1}(J,A) = P(\zeta),
$$
for any $\zeta \in \LieX^c$.

A geometric understanding of the previous construction has been provided in \cite{AGG}. The set $\cB$ can be endowed with a natural structure of symmetric space (see \cite[Th. 3.6]{AGG}) and a principal $\ctG$-bundle $\cY \subset \Aut P$ with connection, which plays the role of the missing complexification of $\ctG$. Then, the complex exponential takes values in $\cY$ and the curve $f_t$ above is just the horizontal lift of the curve in $\cB$.

\section{Deformation of complex structures on pairs}\label{sec:defopairs}

Let $(J,A)\in \cP$ be an integrable pair. This element defines a complex manifold $X$ with K\"ahler form $\omega$ and underlying smooth manifold $M$, together with a holomorphic principal $G^c$-bundle $E\to X$ with underlying smooth bundle $P$ endowed with a reduction $H$.
In this section we define an elliptic complex that encodes the structure of $\cJ\times \cA$ near $(J,A)$,
taking into account the action of $\ctG$ and the subspace $\cP$. With this complex at hand, we construct a slice in $\cJ\times \cA$ that parametrizes nearby structures up to $\ctG^c$-orbits.

Given a $(1,q)$-form $\sigma$ on $X$ (that may take values in a vector bundle), we introduce the following notation. Given $i\geq 1$ and a $(0,i)$-form with values in the holomorphic tangent bundle $\a\in\Om^{0,i}(X,T^{1,0})$, we define a $(0,q+i)$-form $\sigma^\a$ by
\begin{equation}\label{eq:sigmagamma}
\sigma^\gamma= (\sigma(\gamma,\cdot))^{skw}
\end{equation}
where $skw$ denotes the skew-symmetric part of the tensor $\sigma(\gamma,\cdot)$ satisfying :
$$
\forall (X_j)\in (T^{0,1})^{q+i}\; , \; \sigma(\gamma,\cdot)(X_1,\ldots,X_{q+i})=\sigma(\gamma(X_1,\ldots,X_i),X_{i+1},\ldots,X_{q+i}).
$$

\subsection{The elliptic complex}\label{subsec:complex}

Consider the tangent space of $\cJ\times\cA$ at $(J,A)$, naturally identified with
\begin{equation}\label{eq:tangentJA}
T_J\cJ\times T_A\cA \cong \lbrace (\a,\b)\in \Om^{0,1}(T^{1,0})\times \Om^{0,1}(\ad P) \; : \; \om^\a = 0\rbrace 
\end{equation}
where we use the isomorphism $\Omega^{0,1}(\ad P) \to \Omega^1(\ad H)\colon \b \to \b - \b^*$. Given $\gamma \in T_J\cJ$, we write $J_\gamma$ for the almost complex structure induced by $\gamma$ (see e.g. \cite{ku}).

We calculate now the infinitesimal condition for a deformation $(J_\a,A+\beta - \beta^*)$ to be integrable. Denote by
$$
\begin{array}{cccc}
\pr_{J_\a}^{0,2} : & \Om^2(\ad P) &\rightarrow &\Om^{0,2}(\ad P) \\
 & b & \mapsto & \dfrac{1}{4} (b(\cdot,\cdot) + i b(J_\a\cdot,\cdot)+ib(\cdot, J_\a\cdot) - b (J_\a\cdot,J_\a\cdot))
\end{array}
$$
the projection induced by the almost-complex structure $J_\a$. Then, the integrability reads
\begin{equation} \label{eq:integrability}
\begin{split}
\delb \a - [ \a,\a ] & = 0, \\
\pr_{J_\a}^{0,2}F_{A+\beta - \beta^*} & = 0.
\end{split}
\end{equation}
Using the generalized De Rham sequence determined by the covariant derivative $\nabla_A$ on the adjoint bundle (see e.g. \cite{AB})
\begin{equation}
\label{eq:Derham}
\Om^0(\ad H)\lra{d_A} \Om^1(\ad H) \lra{d_A} \Om^2(\ad H) \lra{d_A} ...
\end{equation}
we can write the variation of the curvature of $A + a$ for $a \in \Omega^1(\ad H)$ as
$$
F_{A+a}=F_A+d_A a + [a,a],
$$
and hence the linearisation of \eqref{eq:integrability}, describing the tangent space of $\cP$ at $(J,A)$, is
\begin{equation}
\begin{split}
\delb \a  & = 0,\\
\dbar_{J,A}(\beta) + \dfrac{i}{2}F_A^\gamma & = 0,
\end{split}
\end{equation}
where $\dbar_{J,A} = p_J^{0,1}\nabla_A$ is the induced Dolbeault operator on $\ad P$. The complexified infinitesimal action $P\colon \LieX^c \to T_J\cJ \times T_A\cA$ \eqref{eq:infactcpx} combined with the first order operator
\begin{equation}
\label{eq:deltilde}
\begin{array}{cccc}
\tilde{\del} : & T_J\cJ \times T_A\cA & \rightarrow & \Om^{0,2}(T^{1,0})\times \Om^{0,2}(\ad P) \\
 & (\a , \b) & \mapsto & (\delb \a,\dbar_{J,A}\b + \dfrac{i}{2}F_A^\gamma).
\end{array}
\end{equation}
define our complex of differential operators
\begin{equation}
\label{eq:complex}
\Lie{\ctG}^c  \lra{P}  T_J\cJ \times T_A\cA \lra{\tilde{\del}} \Om^{0,2}(T^{1,0})\times \Om^{0,2}(\ad P).
\end{equation}

\begin{proposition}\label{prop:elcomplex}
The sequence \eqref{eq:complex} is an elliptic complex.
\end{proposition}
\begin{proof}
Note that \eqref{eq:complex} is a complex, as the $\ctG^c$-orbit of $(J,A)$ is contained in $\cP$.
Using the connection $A$, we can write $P$ as
$$
\begin{array}{cccc}
 P_{(J,A)} : & \Lie{\ctG}  & \rightarrow & T_J\cJ \times T_A\cA \\
 & \zeta  & \mapsto & (\dbar \pr(\zeta), -\dbar_{J,A}(A\zeta)- p^{0,1}_J\iota_{\pr(\zeta)}F_A)
\end{array}
$$
and hence 
the ellipticity of the complex is a direct consequence of the ellipticity of the two following complex
$$
\Lie(\cH)^c \lra{\dbar} T_J\cJ \lra{\dbar} \Om^{0,1}(T^{1,0}),
$$
$$
\Om^0(\ad P) \lra{\dbar_A} \Om^{0,1}(\ad P) \lra{\dbar_A} \Om^{0,2}(\ad P).
$$


\end{proof}

\subsection{The Kuranishi method}\label{subsec:slice}

We build now on the elliptic complex \eqref{eq:complex} to define a slice of the $\ctG^c$-orbits in $\cJ\times\cA$ in a neighborhod of $(J,A)$. We use the Kuranishi method \cite{ku}, so we will focus on the novelties in the proof.

The K\"ahler metric on $X$ and the metrics on the spaces $\Om^*(\ad H)$ enable to define the formal adjoints
$P^*$ and $\tilde{\del}_{J,A}^*$. We set $\Delta= PP^* + \tilde{\del}^*\tilde{\del}$ and define
$$
H^1(X,L_{\om}^*):= ker(\Delta)=\lbrace x\in T_J\cJ \times T_A\cA \;:\; \tilde{\del} x=0 \text{ and } P^*x=0 \rbrace,
$$
which is finite dimensional by Proposition \ref{prop:elcomplex} (the notation will be justified in Section~\ref{sec:ap}). Let $\iK$ be the stabilizer of $(J,A)$ in $\ctG$, which is a compact group of finite dimension. As $\iK$ preserves $H^1(X,L_{\om}^*)$, it induces a representation of the complexified group $\iK^c$.

\begin{proposition}
\label{prop:slice}
There exists a map $\Phi$ from a neighborhood $U$ of $0\in H^1(X,L_{\om}^*)$ to a neighborhood of $(J,A)\in \cJ\times\cA$ such that
\begin{enumerate}
 \item $\Phi$ is $\iK$-equivariant, holomorphic, and $\Phi(0)=(J,A)$,
 \item each $\ctG^c$-orbit of integrable complex structure near $(J,A)$ intersects the image of $\Phi$,
 \item if $x$ and $x'$ are in the same $\iK^c$-orbit in $U$, then $\Phi(x)$ and $\Phi(x')$ lie in the same $\ctG^c$-orbit in $\cJ\times\cA$.
\end{enumerate}
\end{proposition}

\begin{proof}
Given $\psi = (\a,\b)\in T_J\cJ \times T_A\cA$, we define operators by
\begin{equation}
\label{eq:operatorQ}
\begin{split}
E_1(\psi) & =\delb \a - [ \a,\a ],\\
E_2(\psi) & =\pr_{J_\a}^{0,2}(F_A+d_A\b+[\b,\b]).
\end{split}
\end{equation}
We set $E=(E_1,E_2)$ and
$$
\Psi=\lbrace \psi\in  T_J\cJ \times T_A\cA \;:\;  E(\psi) = 0,\; P^*(\psi)=0\rbrace,
$$
parameterizing integrable pairs around $(J,A)$. The image of the map that we are looking for will contain a neighborhood of $0\in \Psi$ such that its preimage is an analytic subspace of $H^1(X,L_{\om}^*)$, as in \cite{ku}. Actually, in the ideal case the image is precisely a neighborhood of the origin.

Consider the operator $Q=\tilde{\del} - E$ and the green operator $G$ of $\Delta$ on $T_J\cJ \times T_A\cA$. For any $\psi\in \Psi$
$$
\Delta \psi -\tilde{\del}^*Q( \psi)=0
$$
and applying Green operator to this equality, we obtain
\begin{equation}
\label{eq:harmonicpsi}
\forall \psi\in \Psi, \; \psi - G\tilde{\del}^* Q(\psi)\in H^1(X,L_{\om}^*).
\end{equation}
Then we introduce
$$
\begin{array}{cccc}
 F_1 : & T_J\cJ \times T_A\cA & \rightarrow & T_J\cJ \times T_A\cA \\
 & \psi & \mapsto & \psi - G\tilde{\del}^* Q(\psi)
\end{array}
$$
Using the metrics on $\Om^*(T^{1,0})$ and $\Om^*(\ad H)$, together with the covariant derivatives $\nabla^g$ and $\nabla_A$, we endow  the spaces $\Om^*(T^{1,0})\times \Om^*(\ad H)$ with Sobolev norms $\vert\vert\cdot\vert\vert_k$.
Then the identification \eqref{eq:tangentJA} and the splitting induced by $A$
$$
 0 \to \LieG\simeq \Om^0(\ad H) \lra{\iota} \LieX \lra{\pr} \LieH\subset\Om^0(T^{1,0}) \to 0
$$
induce corresponding norms on $\LieX$ and $T_J\cJ\times T_A\cA$. By \eqref{eq:operatorQ}, $Q$ is a composition of linear and bilinear continuous operators with respect to these norms applied to $\a,\b$ and $d_A\b$, thus $F_1$ extends in a unique way to a continuous operator from the Banach completion $\cB^k$ of $T_J\cJ \times T_A\cA$ with respect to $\vert\vert\cdot\vert\vert_k$ to itself, that we denote also by $F_1$. Its differential at $0$ is the identity, thus by the inverse function theorem, we can define a continuous local inverse of $F_1$ at $0$ from a neghborhood $B$ of $0$ in $\cB^k$ to itself.

Let $F^{-1}_U$ be the restriction of $F^{-1}_1$ to $U=B\cap H^1(X,L_{\om}^*)$. For any $x = F_1(\gamma_x,\beta_x) \in U$ we define
$$
\Phi(x) = (J_{\gamma_x},A + \b_x - \b_x^*).
$$
For each $t \in B$ 
we have that $\Delta F_U^{-1}(t) - \tilde{\del}^* Q F_U^{-1}(t)$ is harmonic, and by elliptic regularity $F_U^{-1}(t)$ is smooth. Then $\Phi$ defines a map
$$
\Phi \colon U \subset H^1(X,L_{\om}^*) \to \cJ\times \cA
$$
such that $\Phi(0)=(J,A)$.

We prove now that $\Phi$ has the desired properties. As $F_1$ is composition of linear and bilinear maps, and the map
\begin{equation}\label{eq:Chernmap}
(\a,\b) \to (J_\gamma,A+\b - \b^*)
\end{equation}
is holomorphic, so is $\Phi$. Note here that the second factor on the RHS corresponds to the Chern connection of $I_{J,A + \b}$. These maps are clearly $K$-invariant 
and hence $\Phi$ satisfies $(1)$. To prove $(2)$, we denote by $\Upsilon$ the inverse of \eqref{eq:Chernmap}, defined only in a neighbourhood of $(J,A)$. Let $(\operatorname{Ker}P)^{\perp}$ be the orthogonal complement of $\operatorname{Ker}P$ \eqref{eq:infactcpx} in the Banach completion of $\Lie \ctG^c$. Using the exponential map \eqref{eq:exp}, we define
$$
\begin{array}{cccc}
F_2 : & (\operatorname{Ker}P)^\perp \times T_J\cJ \times T_A\cA & \rightarrow & \Lie \ctG^c \\
& (\zeta, \psi) & \mapsto & P^*\Upsilon(\exp_{J_\gamma,A_\beta}(\zeta)^{-1}(J_\gamma,A_\beta)
 ),
\end{array}
$$
where $\psi = (\a,\b)$ and $A_\beta = A + \beta - \beta^*$.
The differential of $F_2$ with respect to $\zeta$ at $0$ is $P^*P$, and we proceed as for the map $F_1$ to use the implicit function theorem and obtain a map $\psi \mapsto \zeta(\psi)$ from a neighborhood of $0\in T_J\cJ \times T_A\cA$
to a neighborhood of zero in $(\operatorname{Ker}P)^{\perp}$ satisfying $F_2(\zeta(\psi),\psi) = 0$.
Then, up to shrinking $U$ we have the desired slice properties. 

To prove $(3)$, we assume without lost of generality that $U$ is a ball around the origin with respect to the induced constant metric on $T_J\cJ \times T_A\cA$. By the polar decomposition of $\iK^c$ and the $K$-equivariance of $\Phi$, it is enough to check that $\Phi(x)$ and $\Phi(e^{i\zeta}x)$ lie in the same $\ctG^c$-orbit for any $\zeta \in \Lie \iK$. By convexity of the norm squared along $t \to e^{it\zeta}x$ we have a well defined curve
\[
t \to \Phi(e^{it\zeta}x)
\]
for $t \in [0,1]$. Now, by holomorphicity of $\Phi$ the point $\Phi(e^{it\zeta}x)$ is in the $\cG^c$-orbit of $\Phi(x)$ for all $t$, proving $(3)$.
\end{proof}

\section{Deformation of solutions of the CKYM equations}\label{sec:deformCKYM}

In this section, we apply Proposition~\ref{prop:slice} to the problem of deforming solutions of \eqref{eq:CKYM0}. A triple $(X,L,E)$ will stand for a holomorphic principal $G^c$-bundle $E\rightarrow X$ over
a compact complex manifold $X$, polarized by an ample line bundle $L\rightarrow X$. In this framework, we can regard \eqref{eq:CKYM0} as equations for pairs $(\omega,H)$, where $\omega \in c_1(L)$ is a K\"ahler form and $H$ is a reduction of $E$ to the maximal compact $G$. Fixing $(\omega,H)$, we have associated spaces of compatible pairs $\cP \subset \cJ \times \cA$ acted by the corresponding extended gauge group $\ctG$, with a distinguished point $(J,A)$ given by the holomorphic structures on $X$ and $E$.


\begin{definition}
A \emph{complex deformation of $(X,L,E)$} is a triple $(\cX,\cL,\cE)$ together with a proper holomorphic submersion $\cX \rightarrow \Delta$ with $\Delta$ an open subset of $\CC^p$ containing $0$ such that
\begin{itemize}
\item $\cE$ is a holomorphic principal $G^c$-bundle $\cE\rightarrow \cX$ over
a complex manifold $\cX$, polarized by a line bundle $\cL\rightarrow \cX$,
\item for each $t\in\Delta$, the fiber $\cE_t \to \cX_t$ is by restriction a holomorphic principal $G^c$-bundle and $\cL_t \to \cX_t$ is ample,
\item the triple $(\cX_0,\cL_0,\cE_0)$ is isomorphic to $(X,L,E)$.
\end{itemize}
We say that $(X',L',E')$ is a \emph{small complex deformation} if there exists a complex deformation of $(X,L,E)$ over some $\Delta$ such that $(X',L',E')$ is isomorphic to a fiber over $\Delta$ close to $0$.
\end{definition}

Given a complex deformation $(\cX,\cL,\cE)$ of $(X,L,E)$, by a result of Ehresmann (see e.g. \cite{voi}), we can see $\cX$ as a smooth family of integrable almost-complex structures on $X$. Similarly, we can consider $\cE$ and $\cL$ as smooth families of holomorphic structures over a fixed smooth bundles over $X$. Fixing $(\omega,H)$ on $(X,L,E)$, as before, a small complex deformation $(X',L',E')$ corresponds
to a point $x'$ in $H^1(X,L_{\om}^*)$. As $L'$ and $L$ are isomorphic as smooth bundles
we have $c_1(L) = c_1(L')$, and using Moser's trick we can assume that the almost complex structure $J'$ defining $X'$ is $\om$-compatible.
As $E'$ is seen as a smooth holomorphic structure over a fixed differentiable principal $G^c$ bundle over $X$, it defines a Chern connections $A'$ with respect to $H$, and the pair $(J',A')$ corresponds to an element in $\cP \subset \cJ\times\cA$ close to $(J,A)$. Then, the complexified orbit of $(J',A')$ intersects the image of the $K$-equivariant map $\Phi$ constructed in Proposition \ref{prop:slice}.

We now proceed with the proof of Theorem \ref{theo:finitestab}. As in \cite{sz}, we reduce the proof to a finite dimensional problem. Consider the pull-back K\"ahler form $\Phi^*\omega_\alpha$ on $H^1(X,L_{\om}^*)$, with Hamiltonian action of $K$ and moment map $\Phi^*\mu_\alpha$. We wish now to apply the following proposition due to G. Sz\'ekelyhidi \cite[Proposition 8]{sz}, generalizing a previous result by Donaldson in \cite{D2}.

\begin{proposition}\label{propo:Gabor}
Let $V$ be a complex vector space, endowed with K\"ahler form and a Hamiltonian linear action of a compact group $K$. Let $\mu$ be  a moment map with $\mu(0) = 0$. 
Then, there exists a neighbourhood $U$ of $0 \in V$ such that any polystable $K^c$-orbit intersecting $U$ contains a zero of the moment map in $U$.
\end{proposition}
Here, the K\"ahler form may only be defined in an open $K$-invariant neighborhood of the origin. The polystability of a $K^c$-orbit in $V$ is with respect to the trivial line bundle with trivial linearization. Equivalently, a point $x \in V$ is polystable if and only if $K^c \cdot x \subset V$ is closed. Note that the result follows from Kempf-Ness Theorem when the K\"ahler form is constant over $V$.

The direct application of Proposition \ref{propo:Gabor} does not imply Theorem \ref{theo:finitestab}, as the vanishing of $\Phi^*\mu_\alpha\colon U \to \mathfrak{k}^*$ is a priori a weaker condition than \eqref{eq:CKYM0}. Because of this, we need to modify the slice $\Phi$ built in Section~\ref{subsec:slice}.

\begin{proof}[Proof of Theorem \ref{theo:finitestab}]

Assume that $(X,L,E)$ admits a solution $(\omega,H)$ of \eqref{eq:CKYM0}, that we identify with $(J,A) \in \cP$, a zero of the moment map $\mu_\alpha$. 
We define a map
$$
\mu_\alpha^* \colon \cJ \times \cA \to \Lie \ctG \colon (J,A) \mapsto 4\alpha_1 (\Lambda F_{A} - z) - A^\perp \eta
$$
where $\eta$ is the Hamiltonian vector corresponding to
$$
S_J - \alpha \Lambda^2 \tr F_A \wedge F_A + 4\alpha \Lambda \tr F_A\wedge z - c_z \in C^\infty(X)
$$
Using the exponential map (\ref{eq:exp}), we can deform the map $\Phi$ along complexified orbits so that the map $\mu_{\alpha}^*\circ \Phi$ takes values in $\k$, the Lie algebra of $\iK$. We denote by $\mathfrak{k}^\perp$ the orthogonal of $\mathfrak{k} = \operatorname{Ker}P$ in the Banach completion of $\Lie \cG$. 
Let $W$ be a sufficiently small neighborhood of $0$ in $\mathfrak{k}^\perp$. Consider the following operator:
$$
\begin{array}{cccc}
F_3 : & U \times W & \rightarrow & \mathfrak{k}^\perp  \\
& (x, \zeta) & \mapsto & \Pi\; \mu_{\alpha}^*(\exp_{\Phi(x)}(\zeta)^{-1}(\Phi(x))),
\end{array}
$$
where $\Pi$ is the $L^2$ projection of the Banach completion of $\Lie \ctG$ onto $\mathfrak{k}^\perp$. Then, as $(J,A)$ satisfies \eqref{eq:CKYM}, by \cite[Lemma 4.5]{AGG}, the differential of $F_3$ with respect to $\zeta$ at $0$ is $P^*P$. Applying the implicit function theorem to $F_3$, we obtain a new map that we still denote by $\Phi$ such that the composition by $\mu_{\alpha}^*$ takes value in $\mathfrak{k}$. 
Finally, to obtain a finite dimensional moment map out of $\mu_\alpha^*$, we define a family of non-degenerate pairings (cf. \cite[Remark 4.4]{AGG})
$$
U \to S^2\mathfrak{k}^*\colon x \to (\cdot,\cdot)_x,
$$
where
$$
(\zeta_0,\zeta_1)_x = \int_X \phi_0 \phi_1 \vol_\omega + \int_X (A_x\zeta_0,A_x\zeta_1)\vol_\omega
$$
and $\Phi(x) = (J_x,A_x)$. Then, it is easy to see that
$$
(\mu_\alpha^*\circ \Phi (x),\cdot)_x = \Phi^*\mu_\alpha (x)
$$
and hence the LHS defines a moment map for the $K$-action on $(U,\Phi^*\omega_\alpha)$. The proof follows now applying Proposition \ref{propo:Gabor}.
\end{proof}

We finish this section with some comments on the relation between the sufficient condition for the existence of solutions of \eqref{eq:CKYM0} in Theorem \ref{theo:finitestab} and $\alpha$-K-stability \cite{GF}. Assuming $G^c = SL(r,\CC)$, we can regard $E$ as a holomorphic vector bundle over $X$ with trivial determinant. Following \cite{sz}, it is easy to check that if the small deformation $(X',L',E')$ is $\alpha$-K-polystable then the corresponding point $x \in H^1(X,L_\omega^*)$ is polystable. This follows assuming that $x$ is not polystable and constructing a test configuration (in the sense of \cite{GF}) with smooth central fibre non isomorphic to $(X,L,E)$ and zero Futaki invariant, out of the $1$-parameter subgroup that destabilizes $x$. However, we should emphasize that to the knowledge of the authors there is no strong evidence that $\alpha$-K-stability is even a necessary condition for the existence of solutions of \eqref{eq:CKYM0}.

\section{The complexes $L^*_\om$ and $L^*$}
\label{sec:ap}

This section is rather technical. We first extend the elliptic complex in Section \ref{subsec:complex} to a complex $L^*_\om$. Then, we compare $L^*_\om$ with a standard complex $L^*$ encoding deformations of pairs $(X,E)$, considered previously by Huang \cite{Hu} (see also \cite{Griff}). The aim is to prove Proposition \ref{prop:compare}, that will enable to construct new examples of solutions of \eqref{eq:CKYM0} in Section \ref{sec:example}.

\subsection{Extension of the complex}\label{subsec:complexext}

The elliptic complex (\ref{eq:complex}) encodes infinitesimal deformations of integrable pairs that are compatible with the symplectic form $\om$ on $X$. A generalisation of the compatibility with $\om$ to higher degree forms enables to extend this elliptic complex. We use the notation in \eqref{eq:sigmagamma}.
\begin{definition}\label{def:omcomp}
For $i\geq 1$ a form $\a\in\Om^{0,i}(X,T^{1,0})$ is $\om$-compatible if it satisfies
\begin{equation}\label{eq:omcomp}
\om^\a = 0.
\end{equation}
\end{definition}
Denote by $\Om_{\om}^{0,*}(X,T^{1,0})$ the vector space of $\om$-compatible forms and define $\Om_{\om}^{0,0}(X,T^{1,0})$ to be the space of infinitesimal hamiltonian (complex) vector fields, that is vector fields $y_f$ satisfying
$$
\delb f = \om(y_f,\cdot)
$$
for some smooth complex valued function $f$. For $i\geq 0$, we define
 \begin{equation}
 \label{eq:Lomega}
 L^i_{\om}= \Om^{0,i}_{\om}(X,T^{1,0})\times \Om^{0,i}(X,\ad P).
 \end{equation}
Then, $L_\om^1=T_J\cJ \times T_A\cA$ and the operator $\tilde{\del}$ \eqref{eq:deltilde} extends to an operator
\begin{equation}
\label{eq:extension}
\begin{array}{cccc}
\tilde{\del} : & L^{i}_{\omega} & \rightarrow & L^{i+1}_{\omega} \\
 &(\gamma, \beta) & \mapsto & (\delb \a,\dbar_{J,A}\b + \frac{i}{2}F_A^\a)
\end{array}
\end{equation}
where $\dbar_{J,A} = p_J^{0,i+1}\nabla_A$. We will denote by $H^{i}(X,L_{\om}^*)$ the cohomology groups associated to this complex.

\begin{lemma}\label{lem:extension}
$(L^i_{\om},\tilde{\del})$ is an elliptic complex that extends the complex (\ref{eq:complex}).
\end{lemma}
\begin{proof}
We claim that $\tilde{\del}$ is well defined. This follows from the following formula
$$
\omega^{\dbar \a} = (\dbar\omega)^\a,
$$
which can be easily deduced from the compatibility of $\a$.
Now we prove that $\tilde{\del}\circ\tilde{\del}=0$. From the integrability of $(J,A)$, this is equivalent to
$$
\delb_{J,A}(F_A^\a) + F_A^{\delb \a} = (\delb_{J,A} F_A)^\a = 0,
$$
that follows from the Bianchi identity $d_AF_A=0$.
\end{proof}

\subsection{Comparison with Huang's complex}
In \cite{Hu} (cf. \cite{Griff}), Huang built an elliptic complex encoding deformations of pairs given by a complex manifold $X$ together
with a holomorphic vector bundle $E \rightarrow X$. The aim of this section is to relate the complex built in Section \ref{subsec:complexext} with Huang's complex.

Without harm, we assume that $E$ is a principal $G^c$-bundle over $X$. Consider the holomorphic Atiyah sequence (see e.g. \cite{Fj})
\begin{equation}
 \label{seq:bundles}
0 \rightarrow \ad E \rightarrow T^{1,0}_E \rightarrow T^{1,0} \rightarrow 0,
\end{equation}
where $T^{1,0}_E$ denotes the holomorphic vector bundle $T^{1,0}E/G^c \to X$. Then, Huang's complex is the elliptic complex
$$
L^*=(\Om^{0,*}(T^{1,0}_E), \tilde{\del}_E),
$$
induced by the Dolbeault operator $\tilde{\del}_E$ in $T^{1,0}_E$. Geometrically, it can be identified with the $G^c$-equivariant part of the standard complex $(\Omega^{0,*}(T^{1,0}E),\dbar_E)$ on the total space of $E$.

In \cite{Hu}, the complex $L^*$ was used to construct local moduli spaces of complex structures on the differentiable manifolds underlying $E$ and $X$. To link with the discussion in \cite[Section 1.1]{Hu}, note that, fixing a reduction to $G$, the curvature of the associated Chern connection defines a class in $H^1(X,\Omega^{1,0}(\ad E))$, which corresponds to the sequence \eqref{seq:bundles}. Similarly, our complex $L_{\om}^*$ \eqref{eq:extension} can be used to build local moduli of complex structures with an additional compatibility with $\om$. We will see that, in some special cases, we can relate the cohomology groups of $L^*$ with those of $L^*_\om$.

We need some preliminary results about the cohomology groups $H_{\om}^{0,i}(X,T^{1,0})$ of the complex $(\Om_{\om}^{0,i}(T^{1,0}),\delb)$. Apparently this complex has been considered previously in \cite[Th. 7.1]{Fj}, but our next result does not seem to be in the literature. Note that there are injections
$$
\Om_{\om}^{0,i}(T^{1,0}) \hookrightarrow \Om^{0,i}(T^{1,0})
$$
that induce maps in cohomology
$$
H_{\om}^{0,1}(X,T^{1,0})\rightarrow H^{0,1}(X,T^{1,0}).
$$
%
%

\begin{lemma}
\label{lem:compare2}
The map $H_{\om}^{0,1}(X,T^{1,0})\rightarrow H^{0,1}(X,T^{1,0})$ is surjective if $h^{0,2}(X)=0$ and injective if $h^{0,1}(X)=0$.
\end{lemma}

\begin{proof}
We first show the surjectivity assuming $h^{0,2}(X)=0$. Let $\gamma$ be an element of $\Om^{0,1}(X,T^{1,0})$ and assume that $\gamma\in  ker(\delb)\cap ker(\delb^*)\simeq H^{0,1}(X,T^{1,0})$. We want to prove $\om^\a = 0$. The form $\om^\gamma$ is of type $(0,2)$, so it suffices to prove that it is harmonic. By compactness of $X$, $d \gamma = 0$ and hence
$$
d\om^\a = \om^{d\a} = 0.
$$
Let $e_i$ be a local orthonormal frame for $\om$. Using $d^*\gamma =0$ we also have
$$
d^*\om^{\gamma}= \sum_i  \om(e_i, \nabla_{e_i} \gamma)= d (tr_{\om} \gamma) = 0.
$$
Assume now $h^{0,1}(X)=0$ and let $\alpha \in \Om_{\om}^{0,1}(T^{1,0})$ such that $\delb \alpha =0$ and $[\alpha]=0$ in $H^{0,1}(X,T^{1,0})$. Then there is $V$ in $\Om^{0,0}(X,T^{1,0})$ such that $\delb V = \alpha$. Let $\gamma\in \Om^{0,1}(X)$ be the dual of $V$ with respect to $\om$. From $\om^\alpha =0$, we obtain $\delb\gamma=0$ as in Lemma~\ref{lem:extension}. Using the Hodge decomposition of forms, we deduce that $\gamma=\delb f$ for some complex valued function $f$ and hence $V\in \Om^{0,0}_\om(X,T^{1,0})$.
\end{proof}

We need now to interpret our deformation complex from the point of view of sheaf cohomology. Note that the sequence (\ref{seq:bundles})
is a sequence of holomorphic vector bundles on $X$. Denote by $\cT^{1,0}$, $\cC$ and $\cF$ the sheaves of germs of holomorphic sections of $T^{1,0}$, $T^{1,0}_E$ and $\ad E$, respectively. We then have a short exact sequence
\begin{equation}
 \label{seq:sheaf}
0 \rightarrow \cF \rightarrow \cC \rightarrow \cT^{1,0} \rightarrow 0.
\end{equation}
Denote $\cT^{1,0}_{\om}$ the subsheaf of $\cT^{1,0}$ which consists of sections that are $\om$-compatible, that is, which are local infinitesimal
symplectomorphisms. Then we obtain a short exact sequence by restriction of \eqref{seq:sheaf}
\begin{equation}
 \label{seq:sheafom}
0 \rightarrow \cF \rightarrow \cC_\om \rightarrow \cT^{1,0}_{\om} \rightarrow 0,
\end{equation}
where $\cC_\om$ is the subsheaf of $\cC$ whose projection to $\cT^{1,0}$ lies in $\cT^{1,0}_\om$. This last sequence induces a long exact sequence in \v{C}ech cohomology
\begin{equation}
\label{seq0}
 \begin{array}{ccccccccc}
0 & \rightarrow & H^0(X,\cF) & \rightarrow  & H^0(X,\cC_\om) & \rightarrow & H^0(X,\cT_{\om}^{1,0}) & \rightarrow &\\
& \rightarrow & H^{1}(X,\cF) & \rightarrow  & H^{1}(X,\cC_\om) & \rightarrow & H^{1}(X,\cT^{1,0}_{\om}) & \rightarrow &\\
& \rightarrow & H^{2}(X,\cF) & \rightarrow  & H^{2}(X,\cC_\om) & \rightarrow & H^{2}(X,\cT^{1,0}_{\om}) & \rightarrow & \ldots
\end{array}
\end{equation}
that we wish to use in order to compare the groups $H^{i}(X,\cC_\om)$ with the cohomology of $L_{\om}^*$.
The following Poincar\'e lemma for $\om$-compatible forms of degree $1$ follows easily from the Dolbeault's version of the Poincar\'e Lemma.
\begin{lemma}
 \label{lem:poincare}
Let's assume that $\gamma\in\Om_{\om}^{0,1}(T^{1,0})$ is $\delb$-closed. Then locally it can be written $\delb V$ for a local section $V$
of $\Om_{\om}^{0,0}(T^{1,0})$.
\end{lemma}
We can now proof the Dolbeault isomorphisms for the groups we are ultimately interested in.
\begin{lemma}
\label{lem:dolbom}
We have the following isomorphisms :
\begin{enumerate}
\item $ H^0(X,\cT_{\om}^{1,0}) \simeq H^{0,0}_{\om}(X, T^{1,0})$, $H^1(X,\cT_{\om}^{1,0}) \simeq H^{0,1}_{\om}(X, T^{1,0})$
\item $H^0(X,\cC_\om)\simeq H^0(X,L_{\om}^*)$, $H^1(X,\cC_\om)\simeq H^1(X,L_{\om}^*)$
\end{enumerate}
\end{lemma}

\begin{proof}
 The proof relies on the exactness of the following sequences of sheaves, which follows from Lemma \ref{lem:poincare}.
\begin{equation}
 \label{seq:resol}
0 \rightarrow \cT_{\om}^{1,0} \hookrightarrow \Om_{\om}^{0,0}(T^{1,0}) \lra{\delb} \Om_{\om}^{0,1}(T^{1,0}) \lra{\delb} \Om_{\om}^{0,2}(T^{1,0})
\end{equation}
$$
0 \rightarrow \cC_\om \hookrightarrow \Om_{\om}^{0,0}(T^{1,0}_E)\simeq L^0_{\om} \lra{\tilde{\del}} \Om_{\om}^{0,1}(T^{1,0}_E)\simeq L^1_{\om}  \lra{\tilde{\del}} \Om_{\om}^{0,2}(T^{1,0}_E)\simeq L^2_{\om}.
$$
We discuss $(1)$, as $(2)$ follows in an analogue fashion. Consider the double complex $\check{C}^i(\mathfrak{U},\Om^{0,j}_{\om}(T^{1,0}))$ of \v{C}ech cochains of $\Om^{0,j}_{\om}(T^{1,0})$
with respect to some sufficiently fine cover $\mathfrak{U}$ of $X$. It is endowed with two differentials, the \v{C}ech differential
$\delta$ and the differential on forms induced by $\delb$. These two differential commutes.
Moreover, $(\check{C}^*(\mathfrak{U},\cT_{\om}^{1,0}), \delta)$ computes $H^*(X,\cT_{\om}^{1,0})$ and
$(\Om^{0,*}_{\om}(T^{1,0}), \delb)$ computes $H^*(X,L_{\om}^*)$. Then from the exactness of (\ref{seq:resol}) we deduce the isomorphisms in $(1)$ from standard theory on double complexes (see e.g. \cite{voi}).
\end{proof}

We prove now the main result of this section. For this, note that the group of holomorphic automorphisms $\Aut E$ of $E$ acts on $\Omega^{0,*}(T_E^{1,0})$, inducing a structure of $K^c$-equivariant complex on $L^*$ via the natural inclusion $K \subset \Aut E$.

\begin{proposition}
\label{prop:compare}
There is a $K^c$-equivariant homomorphism
$$
H^1(X,L_{\om}^*)\rightarrow H^1(X,L^*).
$$
Assume that $H^{0,2}(X,\ad E)=0$. Then if $h^{0,1}(X)=0$ this morphism is injective and if $h^{0,2}(X)=0$ it is surjective.
\end{proposition}

\begin{proof}
From standard Dolbeault's Theory we have
$$
H^i(X,\cF)\simeq H^{0,i}(X,\ad E), \qquad H^i(X,\cC)\simeq H^{0,i}(X,T_E^{1,0}),
$$
$$
H^i(X,\cT^{1,0})\simeq H^{0,i}(X,T^{1,0}).
$$
Furthermore, from \eqref{seq:sheaf} and the isomorphisms above we have a long exact sequence
\begin{equation}
\label{seq1}
 \begin{array}{cccccc}
H^{0,0}(X,\ad E) & \hookrightarrow  & H^{0,0}(X,T_E^{1,0}) & \rightarrow & H^{0,0}(X,T^{1,0}) & \rightarrow \\
H^{0,1}(X,\ad E) & \rightarrow  & H^{0,1}(X,T_E^{1,0}) & \rightarrow & H^{0,1}(X,T^{1,0}) & \rightarrow \\
H^{0,2}(X,\ad E) & \rightarrow  & H^{0,2}(X,T_E^{1,0}) & \rightarrow & H^{0,2}(X,T^{1,0}) & \rightarrow ...
 \end{array}
\end{equation}
Then, from the sequence \eqref{seq:sheafom} and Lemma \ref{lem:dolbom} we also obtain
\begin{equation}
\label{seq2}
 \begin{array}{cccccc}
H^{0,0}(X,\ad E) & \hookrightarrow  & H^0(X,L_{\om}^*) & \rightarrow & H_{\om}^{0,0}(X,T^{1,0}) & \rightarrow \\
H^{0,1}(X,\ad E) & \rightarrow  & H^1(X,L_{\om}^*) & \rightarrow & H_{\om}^{0,1}(X,T^{1,0}) & \rightarrow \\
H^{0,2}(X,\ad E) & \rightarrow  & H^2(X,\cC_\om) ...
 \end{array}
\end{equation}
By definition, there is a $K$-equivariant injection $\Om^{0,1}_{\om}(T_E^{1,0}) \hookrightarrow \Om^{0,1}(T_E^{1,0})$ from which we obtain the following $K^c$-equivariant map in cohomology
$$
H^1(X,L_{\om}^*) \rightarrow H^{0,1}(X,T_E^{1,0}) = H^1(X,L^*).
$$
Consider now
$$
 H^{0,1}_{\om}(X,T^{1,0})\rightarrow H^{0,1}(X,T^{1,0}).
$$
From Lemma~\ref{lem:compare2}, if $h^{0,2}(X)=0$ this is a surjective map, and if $h^{0,1}(X)=0$ this is injective. Moreover, in this case $H_{\om}^{0,0}(X,T^{1,0})=H^{0,0}(X,T^{1,0})$. Then, by the assumption $H^{0,2}(X,\ad E)=0$, the result follows, comparing the sequences (\ref{seq1}) and (\ref{seq2}) and doing a standard diagram chasing argument.
\end{proof}

\section{Examples}\label{sec:example}

In this section we apply our theory in two different situations. In any case, we start with a solution $(\omega,H)$ of \eqref{eq:CKYM0} on a triple $(X,L,E)$, given by a holomorphic vector bundle $E$ over a polarized complex manifold $(X,L)$. Then, we use Theorem \ref{theo:finitestab} and Proposition \ref{prop:compare} to construct deformations of the given solution, compatible with $\omega$. The first case is simpler, as we only consider deformations of the bundle $E$, but lead us to new solutions of the equations. For the second case we deform also the polarized manifold, building on the non-trivial solutions of the coupled equations constructed by Keller and T{\o}nnesen-Friedman \cite{KT}.

\subsection{Rank $2$ bundles over product surfaces}
Consider $X=\PP_1\times \PP_1$ polarised by $L = \cO(1,1)$. Here, we use the standard notation $\cO(p,q) = \pi_1^*\cO(p) \otimes \pi_2^*\cO(q)$, where $\pi_j$ denotes the projection on each factor. This has a natural structure of homogeneous K\"ahler manifold endowed with the product of the Fubini-Study metrics $\omega = \omega_1 \times \omega_2$, as considered in \cite[Section 5.2]{AGG}. Given $U_1,U_2$ homogeneous line bundles over $X$ we consider the split rank two bundle
$$
E= U_1 \oplus U_2.
$$
The discussion in \cite[Section 5.2]{AGG} implies that there exists a Hermite-Einstein metric $H$ on $E$ such that $(\omega,H)$ is a solution of the coupled equations for any $\alpha > 0$.

We specify to the cases
$$
U_1 = U_2^* = \cO(1,-1) \quad \textrm{or} \quad U_1 = \cO(-2,0),\; U_2 = \cO(0,-2),
$$
where the first situation corresponds to the reducible point in the moduli of ASD connections in \cite[p. 242]{DK}. In the two examples considered above, we have
$$
\End E = \cO_X \oplus \cO(2,-2) \oplus \cO(-2,2) \oplus \cO_X
$$
and hence by K\"unneth formula and Serre's duality
\begin{align*}
H^1(X,\End E)&= H^0(\PP_1,\cO(2))\oplus H^0(\PP_1,\cO(2)),\\
H^2(X,\End E)& = 0.
\end{align*}
Moreover, using $h^{0,1}(X)=h^{0,2}(X)=0$ and $H^{0,1}(X,T^{1,0})=0$ by Lemma \ref{lem:compare2} and \eqref{seq2} we obtain a  $K^c$-equivariant surjection
$$
H^1(X,\End E) \twoheadrightarrow H^1(X,L_{\om}^*).
$$
The group $K^c$ is given by the following exact sequence
\begin{equation}
\label{eq:groupex1}
1\rightarrow \CC^*\times\CC^*\rightarrow K^c \rightarrow PGL_2(\CC)\times PGL_2(\CC)\rightarrow 1,
\end{equation}
where each $\CC^*$ comes from multiplication on each factor of $E$. Since the $PGL_2(\CC)$ actions lift naturally to $E$, this sequence is split. At the level of Lie algebras, this induces an exact sequence
$$
0\rightarrow H^{0,0}(X,\End E) \to H^{0,0}(X,T_E^{1,0})  \rightarrow  H^{0,0}(X,T^{1,0}) \rightarrow 0.
$$
By proposition~\ref{prop:compare}, we obtain therefore $K^c$-equivariant isomorphisms
$$
H^1(X,\End E) \rightarrow H^1(X,L_{\om}^*)\rightarrow H^1(X,L^*).
$$
Hence, each point in $H^1(X,\End E)$ which is polystable under the $K^c$-action will induce a polystable point in $H^1(X,L_{\om}^*)$ and, by $H^2(X,End E)=0$, a deformation of $(X,L,E)$ of the form $(X,L,E_t)$ with new solutions of the CKYM equations (if close enough to $0$). To find polystable points, we note that the $K^c$-action on $H^1(X,\End E)$ is given by
$$
\begin{array}{ccc}
K^c\times H^{0,1}(X,\End E) & \rightarrow & H^{0,1}(X,\End E)\\
((\lambda_1,\lambda_2,g_1,g_2),(x,y)) & \mapsto & (\lambda_2^{-1}\lambda_1g_1 \cdot x,\lambda_1^{-1}\lambda_2 g_2 \cdot y)
\end{array}
$$
where $g_1\cdot x$ and $g_2\cdot y$ denote the standard linear representation of $PGL_2(\CC)$ on $H^0(\PP_1,\cO(2))$.
Then, the polystable points are the elements $(x,y)$ with $x$ and $y$ non-zero and with simple zeroes, so we conclude that there are stable infinitesimal deformations close to zero.

\subsection{Deformation of Keller--T{\o}nnesen-Friedman solutions}\label{subsec:ex2}

We apply our theory to the solutions constructed in \cite[\S 3]{KT}, that crucially have $\alpha >0$. We need the following result.

\begin{lemma}
\label{lem:defX}
Suppose that $E$ is a line bundle and $h^{0,2}(X)=0$. Then,
\begin{enumerate}
\item any stable point $v\neq 0$ under the action of $Isom_0(g)^c$ in $H^{0,1}(X,TX^{1,0})$ provides a stable point $w\neq 0$ under the $K^c$-action in $H^1(X,L_{\om}^*)$
\item if the above point $v$ correspond to an unobstructed deformation, so does $w$.
\end{enumerate}
\end{lemma}

\begin{proof}
By hypothesis we have $H^{0,2}(X,\End E)=H^{0,2}(X,\cO) = 0$. Now, by proposition~\ref{prop:compare} we have a surjective map
$H^1(X,L_{\om}^*)\rightarrow H^1(X,L^*)=H^{0,1}(X,T_E^{1,0})$ and hence the morphism
$$
H^{0,1}(X,C) \twoheadrightarrow H^{0,1}(X,TX^{1,0})
$$
is surjective by \eqref{seq1}. Thus, by composition there is a surjective morphism
$$
\phi : H^1(X,L_{\om}^*) \twoheadrightarrow H^{0,1}(X,TX^{1,0}).
$$
Considering the homomorphism $\psi : K \rightarrow Isom_0(g)$, by Proposition \ref{prop:compare} the map $\phi$ is equivariant, that is,
\begin{equation}
\label{eq:equiv}
\phi(a\cdot x)=\psi(a)\cdot \phi(x)\qquad \forall \; (a,x)\in  K^c\times H^1(X,L_{\om}^*).
\end{equation}
Assume that there is $v\in H^{0,1}(X,TX^{1,0})$ which is polystable under the $Isom_0^c(g)$-action. By surjectivity of $\phi$ there exists $w'\in H^1(X,L_{\om}^*)$ in the preimage of $v$ by $\phi$ and hence a point $w$ in the closure of the $K^c$-orbit of $w'$ which is polystable. This point is not zero as $\phi$ is equivariant and $v\neq 0$ has a closed $Isom_0^c(g)$-orbit.
\par
To prove $(2)$, note that $\phi(w)$ is in the orbit of $v$. By $Isom_0^c(g)$-equivariance of the Kuranishi slice, as $v$ is unobstructed, so is $\phi(w)$.  Then $w\in H^1(X,L_{\om}^*)$ covers an unobstructed infinitesimal deformation of $X$. By standard theory (see e.g. \cite{ks1}), from $h^{0,2}(X)=0$ we deduce that $E$ and $L$ admit simultaneous deformations that cover the deformation of $X$, that is, $w$ is unobstructed.
\end{proof}

We now apply this result to the examples in \cite[\S 3]{KT}. Let $\Sigma_i$ be a pair of Riemann surfaces of genus $g_i$ and consider
$$X:=\PP(\cL_1\oplus \cL_2) \lra{\pi} \Sigma_1\times \Sigma_2=:S$$
with $\cL_i$ the pull back of a line bundle of positive degree on $\Sigma_i$. In \cite{KT}, Keller and T{\o}nnesen-Friedman showed that for suitable choices of $\cL_i$, there exists a polarization $L$ on $X$ and a line bundle $E \to X$ which admits a solution $(\omega,H)$ of the CKYM equations \eqref{eq:CKYM0}.

To construct new solutions we apply Theorem \ref{theo:finitestab} combined with Lemma \ref{lem:defX}, so we require $h^{0,2}(X)=0$. From the Gysin sequence applied to the above $\PP^1$-fibration we obtain a short exact sequence
$$
0 \rightarrow H^2(S) \rightarrow H^2(X) \rightarrow H^0(S) \rightarrow 0,
$$
where the map $H^2(S) \rightarrow H^2(X)$ is induced by pull-backs. Then for $h^{0,2}(X)$ to vanish it is necessary that $h^{0,2}(S)=0$ and from K\"unneth this forces one of the factors in $S$, $\Sigma_1$ say, to be $\PP^1$.

Assume that $\Sigma_2$ has genus greater than two, so we are in the situation considered in \cite[Ex. 3.5]{KT} (the same calculation applies if we set $\Sigma_2$ to be a torus and consider \cite[Ex. 3.4]{KT}). We assume that $\cL_1$ and $\cL_2$ are the pull-backs of $\cO(1)$ and $\cK^{1/2}$, respectively, where $\cK$ is the canonical line bundle on $\Sigma_2$. By Lemma \ref{lem:defX}, it is enough to consider infinitesimal unobstructed deformations corresponding to closed orbits in $H^{0,1}(X,TX^{1,0})$ under the $Isom_0(g)^c$-action. The metric $g$ is \emph{admissible} and thus its isometry group is given by (see \cite[\S 2]{ap})
$$
Isom_0(g)^c\simeq\CC^*\times \tilde{G}^c,
$$
where the $\CC^*$ factor is generated by the Euler vector field on the fibers, while $\tilde{G}^c$ is a finite cover of $SL(2,\CC)$ that comes from the lift of the complexification of infinitesimal hamiltonian isometries on the surface $S$.
\par
Given a deformation $F_t$ of  $F_0=\cL_1\oplus \cL_2$, we consider deformations of $X$ of the following type
$$
\PP(F_t)\rightarrow S,
$$
parametrized by
$$
H^1(S,\End_0 (\cL_1\oplus \cL_2)))=H^1(S,\cO)\oplus H^1(S,\cL_1^{-1}\otimes \cL_2 )\oplus H^1(S,\cL_2^{-1}\otimes \cL_1).
$$
Here, the subscript $0$ means trace-free endomorphisms. It remains to show that there are closed $Isom_0(g)^c$-orbits in $H^1(S,\End_0(\cL_1\oplus \cL_2)))$. The following identities follow from K\"unneth's formula and Serre duality
\begin{align*}
H^1(S,\cO) & = \(H^{0}(\PP_1,\cO(-2))\otimes H^{0}(\Sigma,\cO)\) \oplus \(H^{0}(\PP_1,\cO)\otimes H^{0}(\Sigma,\cK)\)\\
&= H^{0}(\PP_1,\cO)\otimes H^{0}(\Sigma,\cK),
\end{align*}
and it follows from this that $\CC^*\times \tilde{G}^c$ acts trivially on the left hand side. Now, from
\begin{align*}
H^1(S,\cL_1^{-1}\otimes \cL_2 ) & = \(H^0(\PP_1, \cO(-1))\otimes H^1(\Sigma_2, \cK^{1/2})\)\\
& \oplus \(H^0(\PP_1,\cO(-1))\otimes H^0(\Sigma_2, \cK^{1/2})\)
\end{align*}
we obtain $H^1(S,\cL_1^{-1}\otimes \cL_2 )=0$. Similarly, we can calculate
$$
H^1(S,\cL_2^{-1}\otimes \cL_1)\neq 0,
$$
with $\tilde{G}^c$ action denoted by $(\gamma,y) \mapsto \gamma\cdot y$. Using these isomorphisms, the full action reads
$$
\begin{array}{ccc}
\CC^*\times \tilde{G}^c \times H^1(S,\cO)\oplus  H^1(S,\cL_2^{-1}\otimes \cL_1)  & \rightarrow & H^1(S,\cO)\oplus  H^1(S,\cL_2^{-1}\otimes \cL_1)  \\
((\lambda,\gamma), (x,y)) & \mapsto & (x, \lambda \gamma\cdot y )
\end{array}
$$
and hence the stable points are of the form $(x,0)$ with $x\neq 0$. These correspond to infinitesimal deformations of a line bundle and are of the form $t \mapsto \cL_t$ where $\cL_0=\cK^{1/2}$, that is
$$
\PP(F_t)=\PP(\cO(1)\oplus \cL_t).
$$
Note that $\cL_t=\cL_0\otimes L_t$ where $L_t$ is a flat line bundle. By Lemma \ref{lem:defX}, these deformations induce deformations $(X_t,L_t,E_t)$ of $(X,L,E)$ endowed with solutions of the CKYM equations.

\end{document}